\documentclass[11pt]{article}

\usepackage{amsmath,amssymb,amsthm,setspace,graphicx,array,tikz,url}
\usepackage[text={6.5in,8.4in}, top=1.3in]{geometry}
\newtheorem{theorem}{Theorem}[section]
\newtheorem{proposition}[theorem]{Proposition}
\newtheorem{lemma}[theorem]{Lemma}

\begin{document}
	
\title{
	Highly connected graphs have highly connected spanning bipartite subgraphs
}
	
\author{
	Raphael Yuster
	\thanks{Department of Mathematics, University of Haifa, Haifa 3498838, Israel. Email: raphael.yuster@gmail.com\;.}
}
	
\date{}
	
\maketitle
	
\setcounter{page}{1}
	
\begin{abstract}
	For integers $k,n$ with $1 \le k \le n/2$, let $f(k,n)$ be the smallest integer $t$ such that every $t$-connected $n$-vertex graph has a spanning bipartite $k$-connected subgraph.
	A conjecture of Thomassen asserts that $f(k,n)$ is upper bounded by some function of $k$.
	The best upper bound for $f(k,n)$ is by Delcourt and Ferber who proved that
	$f(k,n) \le 10^{10}k^3 \log n$. Here it is proved that $f(k,n) \le 22k^2 \log n$.
	For larger $k$, stronger bounds hold. In the linear regime, it is proved that for any $0 < c < \frac{1}{2}$
	and all sufficiently large $n$, if $k=\lfloor cn \rfloor$, then $f(k, n) \le 30\sqrt{c} n \le 30\sqrt{n(k+1)}$.
	In the polynomial regime, it is proved that for any $\frac{1}{3} \le \alpha < 1$ and all sufficiently large $n$, if $k = \lfloor n^\alpha \rfloor$, then $f(k ,n) \le 9n^{(1+\alpha)/2} \le 9\sqrt{n(k+1)}$.
\vspace*{3mm}


\end{abstract}

\section{Introduction}\label{sec:intro}

All graphs and digraphs considered here are finite and simple.
For $k \ge 1$, a graph $G$ is {\em $k$-connected} if it has at least $k+1$ vertices, but does not contain a set of $k-1$ vertices whose removal disconnects $G$. Let $\kappa(G)$ be the largest $k$ such that $G$ is $k$-connected.

The study of $k$-connectivity and $\kappa(G)$ are central in graph theory with several classical results relating properties of $G$ to its $k$-connectivity or to the $k$-connectivity of a subgraph of $G$.
One such result is a theorem of Mader \cite{mader-1972} that any graph with $n$ vertices and at least $2kn$ edges has a $k$-connected subgraph. For large $n$, the constant $2$ can be replaced with $\frac{19}{12}$ as proved by Bernshteyn and Kostochka \cite{BK-2016}, improving an earlier result of the author \cite{yuster-2003}, and arriving close to Mader's conjectured optimal value of $\frac{3}{2}$ \cite{mader-1979}.
As any graph has a bipartite subgraph with at least half the number of edges, it follows that any graph with at least, say, $4kn$ edges contains a bipartite $k$-connected subgraph. But what if we require the
subgraph to be {\em spanning}? As observed by Thomassen, if $G$ is $2k-1$ {\em edge-connected} (a graph is
$t$ edge-connected if one must remove at least $t$ edges to disconnect it), then every maximum edge cut
(which, in turn, is a spanning bipartite subgraph) is $k$ edge-connected.
A longstanding conjecture of Thomassen \cite{thomassen-1989} asserts that a similar statement
(replacing $2k-1$ with any function of $k$) should hold for $k$-connectivity. Let us be more formal regarding this problem, as we are addressing it for $k$ that may depend on the number of vertices of $G$.

For integers $k,n$ with $1 \le k \le n/2$, let $f(k,n)$ be the smallest integer $t$ such that every $t$-connected $n$-vertex graph has a spanning bipartite $k$-connected subgraph.
It is clear that $f(1,n)=1$ as shown by taking any spanning tree. It is also clear that the requirement $k \le n/2$ is necessary. Indeed, $\kappa(H) \le n/2$ for any bipartite graph $H$ with $n$ vertices.
Also notice that $f(k,n) \le n-1$ as $K_n$ is $(n-1)$-connected.
It is also not difficult to prove that $f(2,n) = 3$ for all $n \ge 5$ (see Section \ref{sec:concluding} for such a proof). Thomassen's conjecture states that $f(k,n)$ is upper bounded by some function of $k$.

The presently best upper bound on $f(k,n)$ is by Delcourt and Ferber \cite{DF-2015} who 
proved that $f(k,n) \le 10^{10}k^3 \log n$ \footnote{Throughout this paper, $\log n=\log_2 n$.}. While this does not prove Thomassen's conjecture, it does show that
the dependence on $n$ is at most logarithmic. It should be noted that Delcourt and Ferber
explicitly state that their main focus is the dependence on $n$ and not the dependence on $k$.
Nevertheless, the problem is intriguing also for large $k$ that depends on $n$. For example, is it true that
if $G$ has linear connectivity, then it has a spanning bipartite subgraph of reasonable linear connectivity?
Here we give a positive answer to this question and significantly improve the dependence on $k$ and
the absolute constants in all regimes.
\begin{theorem}\label{t:main}
	\,\\
	(a) For all $1 \le k \le n/2$, $f(k,n) \le 22k^2 \log n$.\\
	(b) For all $0 \le \alpha < 1$ and sufficiently large $n$, if $k=\lfloor n^\alpha \rfloor$, then
	$f(k ,n) \le 9n^{(1+\alpha)/2} \le 9\sqrt{n(k+1)}$.\\
	(c) For all $0 < c < \frac{1}{2}$ and sufficiently large $n$, if $k = \lfloor cn \rfloor$,  then $f(k, n) \le 30\sqrt{c}n \le 30\sqrt{n(k+1)}$.
\end{theorem}
\noindent Note that while the constant $c$ in Item (c) covers every possible value, it is only meaningful if
$c < \frac{1}{900}$ since we always have $f(k,n) \le n-1$.

It should be noted that Theorem \ref{t:main} does not shed more light on Thomassen's conjecture as the dependence on $n$ in Case (a) is still logarithmic as in \cite{DF-2015}. On the other hand, it scales the dependence on $k$ from at most quadratic to ever smaller polynomials as $k$ grows beyond $n^{1/3}$.
We also note that while we did not try to optimize the absolute constants, they are reasonably small.

Finally, we note that problem addressed here falls under the following general paradigm 
(see Perarnau and Reed \cite{PR-2017}). Given a graph parameter $\rho$, a value $k$, and a family of graphs
${\mathcal F}$, determine the largest value of $\ell$ such that for every
graph $G$ with $\rho(G) \ge k$ there exists an ${\mathcal F}$-free subgraph $H$ of $G$ with $\rho(H) \ge \ell$.
In our case $\rho=\kappa$ and ${\mathcal F}$  is the family of odd cycles.
Some interesting results following under this paradigm are \cite{CFS-2016,erdos-1968,FKP-2015,KO-2004,PR-2017,thomassen-1983}.

\section{Proof of Theorem \ref{t:main}}\label{sec:proof}

Our proof of Theorem \ref{t:main} is an adaptation of the method of Delcourt and Ferber, but with several changes. In particular, we rely on their lemma on subgraph connectivity of digraphs with high minimum out-degree (Lemma \ref{l:out-deg} below), but we do not rely on Mader's Theorem (the one mentioned in the introduction)
which makes the choices of our maximum counter-example (which we use for contradiction) more economical,
and resulting in meaningful bounds also when $k$ is large.

We first need the following simple and useful lemma from \cite{DF-2015}. For a directed graph $D$ and a vertex
$v \in V(D)$, let $d^+_D(v)$ denote the out-degree of $v$ in $D$ and let $U(D)$ denote the underlying
graph of $D$ which is the undirected graph obtained by ignoring the directions of the edges (each cycle of length $2$, which is possible in $D$, corresponds to a single edge in $U(D)$). 
\begin{lemma}[\cite{DF-2015}]\label{l:out-deg}
	If $D$ is a digraph on at most $n$ vertices and with minimum out-degree
	larger than $(k-1)\log n$, then it contains a subdigraph $D'$ with $\kappa(U(D')) \ge k$
	and where $d^+_{D'}(v)  \ge d^+_{D}(v)-(k-1)\log n$ for all $v \in V(D')$. \qed
\end{lemma}

\begin{lemma}\label{l:union}
	Let $G_1$ and $G_2$ be two vertex-disjoint bipartite $k$-connected subgraphs of a graph $G$. If there are at least $2k-1$ pairwise vertex-disjoint edges of $G$ each having an endpoint in $G_1$ and an endpoint in $G_2$,
	then one can pick $k$ of these edges such that the union of $G_1$ and $G_2$, together with the $k$ picked edges, forms a bipartite $k$-connected subgraph of $G$.
\end{lemma}
\begin{proof}
	Let $(A_1,A_2)$ be a proper bipartition of $V(G_1)$ and let $(A_3,A_4)$ be a proper bipartition of
	$V(G_2)$. Let $F$ be $2k-1$ pairwise vertex-disjoint edges, each with an endpoint in $G_1$ and  an endpoint in $G_2$.
	Let $F_{i,j}$ be the subset of $F$ consisting of the edges with one endpoint in $A_i$ and the other in $A_j$. Note that $F$ is the disjoint union of $F_{1,3},F_{1,4},F_{2,3},F_{2,4}$. So either $|F_{1,3} \cup F_{2,4}| \ge k$ or $|F_{1,4} \cup F_{2,3}| \ge k$. In the former case, we can take $F_{1,3} \cup F_{2,4}$
	and the union of $G_1$ and $G_2$ to form a bipartite graph with proper bipartition $(A_1 \cup A_4, A_2 \cup A_3)$.
	In the latter case, we can take $F_{1,4} \cup F_{2,3}$
	and the union of $G_1$ and $G_2$ to form a bipartite graph with proper bipartition $(A_1 \cup A_3, A_2 \cup A_4)$.
	In any case, the obtained union is $k$-connected as the removal of any $k-1$ vertices keeps $G_1$ and $G_2$ connected and at least one of the connecting edges is still intact so the union remains connected as well.
\end{proof}
\begin{lemma}\label{l:union-singleton}
	Let $G_1$  be a bipartite $k$-connected subgraph of a graph $G$. If there is a vertex $v$ of $G$ outside $G_1$ with at least $2k-1$ neighbors in $G_1$,
	then one can pick $k$ of these neighbors such that the union of $G_1$ and $v$, together with the $k$ edges connecting $v$ to the picked neighbors, forms a bipartite $k$-connected subgraph of $G$. \qed
\end{lemma}
\begin{proof}
	Let $(A_1,A_2)$ be a proper bipartition of $V(G_1)$ where, without loss of generality, $v$ has at least $k$ neighbors in $A_1$. Taking $G_1$ together with $v$ and the $k$ edges connecting $v$ to $A_1$
	gives a bipartite graph with bipartition $(A_1,A_2 \cup \{v\})$ which is clearly $k$-connected.
\end{proof}

We begin by proving Case (a) of Theorem \ref{t:main}. We then show how to modify the proof to obtain the other cases.
\begin{proof}[Proof of Theorem \ref{t:main}, Case (a)]
	Recall that we must prove that $f(k,n) \le 22k^2 \log n$.
	We assume $k \ge 3$ (as $k=1$ is trivial and $k=2$ follows form Proposition \ref{prop:1}).
	Let $s=\lfloor 22k^2 \log n \rfloor$. We may assume that $s \le n-1$ since $f(k,n) \le n-1$.
	Let $G$ be an $s$-connected graph with $n$ vertices. We prove that $G$ has a bipartite spanning $k$-connected subgraph.
	
	Let $d=\lfloor s/4 \rfloor \ge k$.
	Consider a partition of the vertex set of $G$ into parts $V_1,\ldots,V_t$ such that each $V_i$ is either a singleton or else $|V_i| \ge d$ and $G[V_i]$ has a bipartite spanning subgraph $B_i$ that is $k$-connected
	(for convenience, when $V_i$ is a singleton, let $B_i$ denote the graph with a single vertex; it is ``bipartite'' with one part of the bipartition being empty).
	Hereafter we shall consider a partition $P=\{V_1,\ldots,V_t\}$ where $t$ is minimum and show that we must have $t=1$, concluding the proof. So, assume that $t > 1$.
	
	Suppose $V_i \in P$ is not a singleton (so $n > |V_i| \ge d$). Let $M_i$ be a maximum matching of the edge cut $(V_i,V(G) \setminus V_i)$. We must have $|M_i| \ge d$ as otherwise the $2|M_i| \le 2d-2 \le s-1$ vertices of the endpoints of $M_i$, once removed from $G$, disconnect the remaining vertices of $V_i$ from the remaining vertices of $V(G) \setminus V_i$, contradicting the $s$-connectivity of $G$.
	Let $S_i \subseteq M_i$ be such that for any $j \neq i$, there is a single edge incident with $V_j$, if such an edge exists. By Lemma \ref{l:union},
	$|S_i| \ge |M_i|/(2k-2) \ge d/(2k-2)$ as otherwise $t$ would not be minimum.
	
	Next, suppose that $V_i = \{v\} \in P$ is a singleton. We say that $V_i$ is of {\em type $\alpha$} if $v$ has at least $3d$ neighbors in $V$, where each of them is a singleton part of $P$. Otherwise, we say that $V_i$ is of {\em type $\beta$}. Now, if $V_i$ is of type $\alpha$, let $S_i$ be a set $3d$ edges connecting $v$ to $3d$ singleton parts of $P$. Suppose now that $V_i$ is of type $\beta$.
	As the degree of $v$ in $G$ is at least $s$, we have that at least $s-3d$ edges connect $v$ to non-singleton parts of $P$. By Lemma \ref{l:union-singleton} and as $t$ is minimum,
	at most $2k-2$ such edges connect $v$ to the same non-singleton part of $P$. Thus, there is a set
	$S_i$ of edges connecting $v$ to non-singleton parts of $P$ such that no two edges of $S_i$ are incident with the same part and $|S_i| \ge (s-3d)/(2k-2) \ge d/(2k-2)$.
	
	To summarize, for each part $V_i \in P$ we have chosen a set of edges $S_i$ incident to that part and connecting it to other parts, such than no other part $V_j$ has more than one vertex incident with $S_i$.
	If $V_i$ is non-singleton, then $S_i$ is a matching and if $V_i$ is a singleton, then $S_i$ is a star.
	In all cases, $t-1 \ge |S_i| \ge d/(2k-2) \ge (s-3)/(8k-8) > s/8k$ and if $S_i$ is a singleton of type $\alpha$, then $|S_i| = 3d$.
	
	Independently for each $i$, take a random red-blue proper vertex coloring of the bipartite graph $B_i$ (so either one part of the bipartition of $B_i$ is red and the other part is blue, or vice versa).
	Once making these $t$ choices, let $T_i \subseteq S_i$ be those edges whose endpoints have different colors. Observe that the union of all $B_1,\ldots,B_t$ and the edge sets $T_1,\ldots,T_t$ is a bipartite spanning subgraph of $G$. As $|T_i|$ has distribution ${\rm Bin}(|S_i|,\frac{1}{2})$ we have
	by Chernoff's inequality,
	\begin{equation}\label{e:1}
	\Pr \left[ |T_i| < \frac{|S_i|}{4}\right] = \Pr \left[ |T_i|-{\mathbb E}[|T_i|] < -\frac{|S_i|}{4}\right] < e^{-\frac{2(|S_i|/4)^2}{|S_i|}}=e^{-\frac{|S_i|}{8}} \le e^{-\frac{s}{64k}}\;.
	\end{equation}
	Since $t \le n$ and since $s > 64k \ln n$, we have from \eqref{e:1} and the union bound that with positive probability, $|T_i| \ge |S_i|/4$ for all $1 \le i \le t$.
	Hereafter, we shall assume that this is indeed the case.
	
	We construct a directed graph $D$ on vertex set $[t]$ where $(i,j)$ is an edge whenever $T_i$ is
	incident with a vertex of $V_j$. Notice that the out-degree of $i$ is precisely
	$|T_i| \ge |S_i|/4 > s/32k$. Since $s \ge 32k(k-1)\log n$, we have that $|T_i| > (k-1)\log n$
	so by Lemma \ref{l:out-deg}, $D$ has a subdigraph $D'$ on vertex set $V(D') \subseteq [t]$ such
	that $U(D')$ is $k$-connected and for each $i \in V(D')$ we have $d^+_{D'}(i) \ge |T_i| - (k-1)\log n$.
	For $i \in V(D')$, let $T_i^* \subseteq T_i$ be the set of edges in $T_i$ incident with some $V_j$ where $j \neq i$ and $j \in V(D')$. Notice that $|T_i^*| = d^+_{D'}(i) \ge |T_i| - (k-1)\log n$.
	
	Consider the bipartite graph $B$ obtained by taking the union of all $B_i$ for $i \in V(D')$ and all edge sets $T_i^*$ for $i \in V(D')$. First notice that $|V(B)| \ge d$. Indeed, either some $i \in V(D')$ is such that $V_i$
	is not a singleton (so already $|V_i| \ge d$), or else for all $i \in V(D')$ we have that $|V_i|$ is a singleton. But in the latter case, all these singletons must be of type $\alpha$ (if $V_i$ is of type $\beta$, then any edge $(i,j)$ in $D'$ is such that $V_j$ is non-singleton).
	However, recall that if $V_i$ is of type $\alpha$, then $|S_i| \ge 3d$ which means that
	$|T_i| \ge 3d/4$. So
	\begin{equation}\label{e:2}
	|T_i^*| = d^+_{D'}(i) \ge |T_i| - (k-1)\log n \ge \frac{3d}{4}-\frac{s}{32k} \ge \frac{3d}{4}-\frac{4d+4}{32k} \ge \frac{d}{2}
	\end{equation}
	and thus $B$ is a bipartite graph with minimum degree at least $d/2$ so it has at least $d$ vertices.
	
	 We will prove that $B$ is $k$-connected. Note that once we do that, we arrive at a contradiction, as the number of parts of the partition obtained by replacing all the parts of $P$ of the form $V_i$ with $i \in V(D')$ with the single element $V(B)$, is smaller than $t$.
	 
	To prove that $B$ is $k$-connected, consider some set $K$ of $k-1$ vertices of $B$. We must show that
	$B^*$, the graph remaining from $B$ after removing $K$, remains connected. Suppose $x \in V_i$ and $y \in V_j$ (possibly $i=j$) are two vertices of $B^*$. We must show that there is a path in $B^*$ connecting them. Now, if $i=j$, then $B_i[V_i \setminus K]$ is connected since $B_i$ is $k$-connected. Hence, there is
	a path between $x$ and $y$ entirely inside $B_i[V_i \setminus K]$. Assume therefore that $i \neq j$.
	Let $L=\{\ell \in V(D') ~|~ V_\ell \cap K \neq \emptyset\}$.
	As $U(D')$ is $k$-connected, removing the vertices of $L$ from $D'$ keeps $U(D'[V(D') \setminus L])$ connected. It therefore suffices to show that there is a (possibly trivial) path connecting $x$ to some vertex of $V_m$ where $m \notin L$ and there is a (possibly trivial) path connecting $y$ to some vertex
	of $V_{m^*}$ where $m^* \notin L$. As the two claims are identical, we prove it for $x$.
	Notice that the claimed path trivially exists if $i \notin L$. If, however $i \in L$ (namely, some of the vertices of $K$ are in $V_i$), then $V_i$ cannot be a singleton (as it contains both $x$ and a vertex of $K$). So, $T_i^*$ is matching of size at least $k$ (since $U(D')$ is $k$-connected, its minimum degree is at least $k$). Hence, at least one edge of $T_i^*$ is not incident with $K$ nor with $V_\ell$ where $\ell \in L$. Such an edge is of the form $uv$ where $u \in V_i \setminus K$ (possibly $u=x$) and $v \in
	V_m$ with $m \notin L$. As there is a path in $B_i[V_i \setminus K]$ between $x$ and $u$, the claim follows.
\end{proof}

\begin{proof}[Proof of Theorem \ref{t:main}, Case (b)]
	Recall that we must prove that for all $0 \le \alpha < 1$ and sufficiently large $n$,
	if $k=\lfloor n^\alpha \rfloor$, then $f(k ,n) \le 9n^{(1+\alpha)/2}$.
	The proof is very similar to Case (a). We use the same notation and outline the differences.
	Set $s=\lfloor 9n^{(1+\alpha)/2} \rfloor$ and let
	$k=\lfloor n^\alpha \rfloor$. We set $d=\lfloor s/4 \rfloor \ge k$ and use a partition
	$P=\{V_1,\ldots,V_t\}$ where $t$ is minimum, as in Case (a), aiming to prove that $t=1$. Assume, to the contrary, that $t > 1$. Let $t^*$ denote the number of non-singleton parts (i.e. those of size at least $d$) and observe that $t^* \le n/d \le 4n/(s-3) < n^{(1-\alpha)/2}/2$.
	
	Suppose $V_i  \in P$ is not a singleton and recall that we have defined $M_i$ to be a maximum matching 
	of the edge cut $(V_i,V(G) \setminus V_i)$ and that as in Case (a) we have $|M_i| \ge d$. By Lemma \ref{l:union}, at most
	$(2k-2)t^*$ edges of $M_i$ are incident with non-singleton parts, so after removing them
	from $M_i$ we obtain a subset $S_i$ with $|S_i| \ge d-(2k-2)t^* \ge 2n^{(1+\alpha)/2}-2n^\alpha n^{(1-\alpha)/2}/2 \ge n^{(1+\alpha)/2}$. Notice that this definition of $S_i$ is slightly different from
	the  one used in Case (a), as we may take advantage of the fact that there are many singleton parts.

	Consider now a singleton part $V_i$ and recall that if $V_i$ is of type $\alpha$, then $S_i$ is the set of edges of a star connecting $V_i$ only to singleton parts and $|S_i|=3d$.
	We claim that there are no singletons of type $\beta$. 
	Suppose that $V_i$ is of type $\beta$, i.e., at least $s-3d$ edges connect it
	to non-singleton parts. However, by Lemma \ref{l:union-singleton} at most $(2k-2)t^*$ edges connect it to non-singleton parts, which is impossible since $(2k-2)t^* < n^{(1+\alpha)/2} < s-3d$.
    
    We observe that as in \eqref{e:1}, we have that
    $\Pr \left[ |T_i| < |S_i|/4 \right]< e^{-\frac{|S_i|}{8}}$ and since $|S_i|/8 = \omega( \ln n)$,
    we have that $|T_i| \ge \frac{|S_i|}{4}$ holds for all $1 \le i \le t$ with positive probability,
    so we assume this is the case. Thus, we have that in the directed graph $D$, the out-degree of $i$ 
    is $|T_i| \ge |S_i|/4 \ge n^{(1+\alpha)/2}/4 > (k-1)\log n$. We therefore apply Lemma \ref{l:out-deg} to obtain $D'$ as in Case (a). The analogue of \eqref{e:2}, which, recall, is only applied when $V_i$ is of type $\alpha$, now becomes
    $$
    |T_i^*| = d^+_{D'}(i) \ge |T_i| - (k-1)\log n \ge \frac{3d}{4}-n^{\alpha}\log n \ge \frac{3d}{4}-\frac{d}{4} \ge \frac{d}{2}
    $$
    so, as in Case (a), $|V(B)| \ge d$, the new partition has fewer parts than $P$, and $B$ is $k$-connected, contradicting the minimality of $t$.
\end{proof}

Prior to proving Case (c) of Theorem \ref{t:main}, we require the following variant of Lemma \ref{l:out-deg}.
\begin{lemma}\label{l:out-deg-linear}
	Let $0 < c < \frac{1}{2}$. For all sufficiently large $n$, if $D$ is a digraph on at most $n$ vertices and with minimum out-degree at least $3c\log(1/c)n$, then it contains a subdigraph $D'$ with $\kappa(U(D')) \ge cn$
	and where $d^+_{D'}(v)  \ge d^+_D(v) - 2c\log(1/c) n$ for all $v \in V(D')$.
\end{lemma}
\begin{proof}
	Let $\gamma = 3c\log(1/c)$. If $\gamma \ge 1$ or if $\kappa(U(D)) \ge cn$, there is nothing to prove.
	So, assume that $\gamma < 1$ and that $\kappa(U(D)) < cn$. Delete a separating set of size at most $cn$. The smallest component, denoted by $A$, has fewer than $n/2$ vertices and for any $v \in V(A)$, every
	out-neighbor of $v$ is either in $V(A)$ or in the separating set that we removed.
	Hence, $d^+_{A}(v) \ge d^+_D(v) - cn$. We repeatedly apply this step, and note that this process must terminate. However, after step $r=\lceil \gamma/2c \rceil$, we are left with a component which
	consists of fewer than $n/2^r$ vertices where each of these vertices has out-degree at least
	$\gamma n - rcn$. However, this is impossible since
	$$
	  \gamma - rc = \gamma - c \left\lceil \frac{\gamma}{2c} \right\rceil \ge \frac{\gamma}{2}-c \ge \frac{\gamma}{6} \ge
	  c^{3/2} = 2^{1.5\log c} = 2^{-\frac{\gamma}{2c}} \ge 2^{-r}\;.
	$$
	Thus, the process ends after at most $r-1$ steps, finding the desired $D'$ and we have
	$d^+_{D'}(v)  \ge d^+_D(v) - (r-1)cn \ge d^+_D(v) - (\gamma/2)n \ge d^+_D(v) - 2c\log(1/c) n$
	for all $v \in V(D')$.
\end{proof}
\begin{proof}[Proof of Theorem \ref{t:main}, Case (c)]
	Recall that we must prove that for all $0 < c < \frac{1}{2}$ and sufficiently large $n$, if $k = \lfloor cn \rfloor$, then $f(k, n) \le 30\sqrt{c}n$.
	The proof is similar to Case (b). We use the same notation and outline the differences.
	Let $c^*=30\sqrt{c}$ and observe that Case (c) holds vacuously when $c \ge 1/900$, hence we assume that $c < 1/900$. let $k=\lfloor cn \rfloor$, $s=\lfloor c^*n \rfloor$, $d=\lfloor s/4 \rfloor$. We now have
	the number of non-singleton parts is $t^* \le n/d \le 4n/(s-3) < 5/c^*$.
	
	If $V_i  \in P$ is not a singleton we have, as in Case (b), that
	$$
	|S_i| \ge d-(2k-2)t^* \ge n(c^*/5 - 10c/c^*) \ge n\left( 6\sqrt{c}-\sqrt{c}/3 \right) \ge 5\sqrt{c}n\;.
	$$
	Consider now a singleton part $V_i$ and recall that if $V_i$ is of type $\alpha$, then $S_i$ is a star
	connecting $V_i$ only to singleton parts and $|S_i|=3d$.
	We again claim that there are no singletons of type $\beta$. 
	Suppose that $V_i$ is of type $\beta$, i.e., at least $s-3d$ edges connect it
	to non-singleton parts. However, by Lemma \ref{l:union-singleton} at most $(2k-2)t^*$ edges connect it to non-singleton parts, which is impossible since
	$(2k-2)t^* < n(10c/c^*) < n(c^*/5) < d \le s-3d$.
	
	As in Case (b), since $|S_i|/8 = \omega( \ln n)$, we may assume that $|T_i| \ge \frac{|S_i|}{4}$ holds for all $1 \le i \le t$. Thus, we have that in the directed graph $D$, the out-degree of $i$ 
	is
	$$
	|T_i| \ge \frac{|S_i|}{4} \ge \frac{5}{4}\sqrt{c}n \ge 3c\log(1/c)n
	$$
	where in the last inequality we have used that $c < 1/900$. We therefore apply Lemma \ref{l:out-deg-linear} to obtain $D'$ with  with $\kappa(U(D')) \ge cn$
	and where $d^+_{D'}(i)  \ge |T_i| -2c\log(1/c) n$ for all $i \in V(D')$. The analogue of \eqref{e:2}, which, recall, is only applied when $V_i$ is of type $\alpha$, now becomes
	$$
	|T_i^*| = d^+_{D'}(i) \ge |T_i| -2c\log(1/c) n \ge \frac{3d}{4}-n\sqrt{c} \ge \frac{3d}{4} - nc^*/30 \ge \frac{3d}{4}-\frac{d}{4} \ge \frac{d}{2}
	$$
	so, as in Case (a), $|V(B)| \ge d$, the new partition has fewer parts than $P$, and $B$ is $k$-connected, contradicting the minimality of $t$.	
\end{proof}

\section{Concluding remarks}\label{sec:concluding}

As mentioned in the introduction, it is not difficult to prove that $f(2,n)=3$ for all $n \ge 5$
($f(2,4)=2$) but as we have not found a proof in the literature, we present one.
\begin{proposition}\label{prop:1}
	$f(2,n) = 3$ for all $n \ge 5$.
\end{proposition}
\begin{proof}
	The lower bound follows by considering a cycle of length $n$ when $n \ge 5$ is odd, and a cycle of length $n-1$
	with an additional vertex connected to two non-adjacent cycle vertices when $n \ge 6$ is even.
	These graphs are $2$-connected and have no $2$-connected spanning bipartite subgraphs.
	
	For the upper bound, suppose $G$ has $n$ vertices and is $3$-connected.
	Let $G^*$ be a subgraph of maximum order that is $2$-connected and bipartite.
	Notice that $G^*$ is nonempty since a $3$-connected graph has an even cycle.
	We claim that $|V(G^*)|=n$. Assuming the contrary, let $A$ and $B$ denote the two parts of the bipartition of $G^*$ and let $G'$ be the subgraph of $G$ induced by the vertices not in $G^*$.
	Note that it is possible that $G'$ is not connected (nor bipartite), so let $X$ be the set of vertices of some connected component of $G'$ and let $T$ be a spanning tree of $G[X]$. Let $C$ and $D$ denote
	the two parts of the bipartition of $T$.
	
	Any vertex $x \in X$ cannot have two neighbors in $A$ as otherwise we can add
	to $G^*$ the vertex $x$ and two edges connecting it to $A$, forming a larger $2$-connected bipartite subgraph. Similarly, we cannot have two neighbors of $x$ in $B$.
	Thus, every $x \in X$ has at most one neighbor in $A$ and at most one neighbor in $B$.
	In particular, $X$ is not a singleton, as otherwise there is a vertex of degree $2$ (namely $x$) in $G$
	which is impossible as $G$ is $3$-connected.
	
	Suppose that $x$ has a neighbor $a \in A$ and a neighbor $b \in B$.
	We claim that (i) for any $y \in X$ with $y \neq x$ such that both $x$ and $y$ are in the same part
	of the bipartition of $T$ (either both in $C$ or both in $D$), $y$ cannot have any neighbor in
	$(A \cup B) \setminus \{a,b\}$.
	Indeed suppose $y$ has such a neighbor, say $a' \in A$.
	Then we can add to $G^*$ the unique (even length) path of $T$ connecting $x$ and $y$, together with the edges $xa$ and $ya'$, forming a larger $2$-connected bipartite subgraph.
	Similarly, we claim that (ii) for any $y \in X$ such that $x$ and $y$ are in opposite parts (one in $C$ and one in $D$), $y$ cannot have any neighbor in $(A \cup B) \setminus \{a,b\}$.
	Indeed suppose $y$ has such a neighbor, say $b' \in B$.
	Then we can add to $G^*$ the unique (odd length) path of $T$ connecting $x$ and $y$, together with the edges $xa$ and $yb'$, forming a larger $2$-connected bipartite subgraph.
	
	We also claim (iii) that we cannot have a matching of size three between $X$ and $G^*$ (this is similar to the proof of Lemma \ref{l:union}). To see this, suppose that $x_iv_i$ for $i=1,2,3$ form a matching of size  three between $X$ and $G^*$
	where $v_1,v_2,v_3 \in V(G^*)$ and $x_1,x_2,x_3 \in X$.
	If $v_1,v_2,v_3$ are in the same part of $G^*$ (either all in $A$ or all in $B$), then
	w.l.o.g.\ $x_1$ and $x_2$ are in the same part of $T$, so there is a path of even length in 
	$T$ connecting $x_1$ and $x_2$, and we may add that path
	and the edges $x_1v_1$, $x_2v_2$ to $G^*$ and obtain a larger $2$-connected bipartite graph.
	Otherwise, assume $v_1,v_2$ are in the same part of $G^*$ and $v_3$ is in the other part of $G^*$.
	If $x_1$ and $x_2$ are in the same part of $T$, then the same argument follows.
	Otherwise, $x_1$ and $x_2$ are in different parts of $T$.
	W.l.o.g.\ $x_3$ and $x_2$ are in distinct parts of $T$ and recall that $v_2$ and $v_3$ are in distinct parts of $G^*$. Then we can add the odd length path of $T$ between $x_2$ and $x_3$ 
	and the edges $v_2x_2$, $v_3x_3$ to $G^*$ and obtain a larger $2$-connected bipartite graph.
	
	Having shown (i)-(iii), we claim that we can disconnect $G$ by removing two vertices, which is a contradiction. Take a maximum matching $M$ between $X$ and $A \cup B$. If $M$ consists of one edge, we can
	disconnect $G$ by removing the endpoints of that single edge. Hence, by (iii), it consists of precisely two edges, say $M=\{xv_1,yv_2\}$ where $x,y \in X$.
	If $v_1$ is the only neighbor of $X$ in $A \cup B$ and $v_2$ is the only neighbor of $y$ in $A \cup B$,
	then we can remove $v_1,v_2$ from $G$ thereby separating $X$ from the rest of the graph (recall that
	$M$ is a maximum matching). Otherwise, w.l.o.g., $x$ has two neighbors in $A \cup B$, one of which is
	$v_1$ and the other denoted by $v_3$ (possibly $v_3=v_2$). If $y$ also has two neighbors in $A \cup B$,
	then by (i) and (ii) it must be that these neighbors are also $v_1$ and $v_2=v_3$.
	We can then remove $v_1,v_2$ from $G$ and disconnect $X$. Otherwise,
	$v_2$ is the only neighbor of $y$ in $A \cup B$. Now, if $v_2=v_3$ we can delete $v_1$ and $v_2$ from $G$ thereby separating $X$ from the rest of the graph. Otherwise, we can delete $x$ and $v_2$ thereby separating $y$ from the rest of the graph.
\end{proof}

Recall that if $G$ is $2k-1$ edge-connected, then every maximum edge cut
is $k$ edge-connected. It is easy to generalize this observation and show that if
the edge connectivity of $G$ is larger than $r(k-1)/(r-1)$, then $G$ has a spanning $r$-chromatic
subgraph that is $k$ edge-connected. Indeed, any maximum $r$-cut validates this fact.
Stated otherwise, by allowing more colors for the spanning subgraph, we can keep the edge-connectivity
close to its original value. The following proposition shows that the same holds for
the case of $k$-connectivity in the linear regime.
\begin{proposition}\label{p:random}
	Let $0 < c < c'<1$. There is a constant $r=r(c,c')$ such that for all $n$ sufficiently large, if
	$G$ has $n$ vertices and is  $\lfloor c'n \rfloor$-connected, then $G$ has an $r$-colorable spanning subgraph that is $\lfloor cn \rfloor$-connected.
\end{proposition}
\begin{proof}
	Throughout the proof we assume that $cn$ and $c'n$ are integers as this does not affect the
	correctness of the proposition. Let $t=c'n$ and let $G$ be a $t$-connected graph on $n$ vertices.
	Let $r=r(c,c') \ge 3$ be a constant whose existence is shown later. Consider an $r$-coloring of $V(G)$ where each vertex independently chooses its color uniformly at random. Let $G'$ be the
	$r$-colorable spanning subgraph of $G$ consisting of all edges whose endpoints receive distinct colors.
	We will prove that with positive probability, $G'$ is $cn$-connected.
	
	By Menger's Theorem, for any two distinct vertices $a,b \in V(G)$, there is a set ${\cal P}_{a,b}$ of $t-1$
	paths connecting $a$ and $b$, each of length at least two, such that any two paths in ${\cal P}_{a,b}$ are internally vertex-disjoint.
	We say that a path $P \in {\cal P}_{a,b}$ {\em survives} if $P$ is present in $G'$.
	It suffices to prove that the probability that fewer than $cn$ paths of ${\cal P}_{a,b}$ survive is less than $1/n^2$ as then, by the union bound and Menger's theorem, $G'$ is $cn$-connected with positive probability.
	
	Fix two vertices $a$ and $b$ and fix a coloring of these two vertices (it is possible for $a$ and $b$ to have the same color).
	Consider some path $P \in {\cal P}_{a,b}$ and let $\ell=|E(P)|-1$ denote the number of its internal vertices.
	The number of possible colorings of these internal vertices is
	$r^\ell$ and at least $(r-1)^{\ell-1}(r-2) \ge (r-2)^\ell$ of these colorings yield a surviving path.
	Hence, the probability that $P$ survives is at least $(1-2/r)^\ell$.
	If $X_P$ denotes the indicator variable for survival and $X = \sum_{P \in {\cal P}_{a,b}} X_P$ is the number of surviving paths,
	$$
	{\mathbb E}[X] = \sum_{P \in {\cal P}_{a,b}} {\mathbb E}[X_P] \ge
	\sum_{P \in {\cal P}_{a,b}} \left(1-\frac{2}{r}\right)^{|E(P)|-1}\;.
	$$
	The total number of internal vertices in all paths of $P \in {\cal P}_{a,b}$ is at most $n-2$, so
	the average number of internal vertices is at most $(n-2)/(t-1)$. By the last inequality and the AM-GM inequality we have
	$$
	{\mathbb E}[X] \ge (t-1)\left(1-\frac{2}{r}\right)^{(n-2)/(t-1)}\;.
	$$
	Let $c^*=(c+c')/2$. Recalling that $t=c'n$ we have from the last inequality, that for all
	$r$ sufficiently large as a function of $c'$ and $c^*$,
	$$
	{\mathbb E}[X] \ge c^*n\;.
	$$
	Notice, however, that $X$ is the sum of $t-1$ independent indicator variables, as the paths are internally vertex-disjoint. Hence, by Chernoff's inequality (in particular, Corollary A.1.7 of \cite{AS-2004})
	$$
	\Pr[ X - {\mathbb E}[X] < (c-c^*)n ] < e^{-2(c^*-c)^2n^2/t} < \frac{1}{n^2}
	$$
	implying that $\Pr[X < cn] < 1/n^2$. As the last statement holds for any choice of the colors of $a$ and $b$, we have that the probability that the number of surviving paths in ${\cal P}_{a,b}$ (regardless of the colors of $a$ and $b$) is smaller than $cn$ is less than $1/n^2$.
\end{proof}

\section*{Acknowledgment}
I thank both reviewers for insightful suggestions.

\end{document}